\documentclass[11pt,a4paper]{article}
\usepackage[margin=0.7in]{geometry}
\usepackage{bm} 
\usepackage[utf8]{inputenc}
\usepackage{tabularx}
\usepackage{caption}
\usepackage{subcaption}
\usepackage{booktabs}
\usepackage{comment}
\usepackage{algorithm}
\usepackage{textcomp} 
\usepackage{algpseudocode}
\usepackage{abstract}
\usepackage{enumitem}
\usepackage{xcolor}
\usepackage{amsthm}
\newtheorem{proposition}{Proposition}
\usepackage{amsmath,amssymb}

\usepackage{cite}

\usepackage{float}
\usepackage{nameref,hyperref}

\usepackage[right]{lineno}

\usepackage{changepage}

\usepackage[labelfont=bf]{caption}

\usepackage{color}

\definecolor{Gray}{gray}{.25}

\usepackage{graphicx}

\usepackage[english]{babel}
\usepackage[autostyle]{csquotes}
\usepackage{longtable}
\begin{document}


\title{\Large \textbf{A Semigroup-Aware Kernel Neural Operator with Latent Exponential Integration for Forced PDEs}
}
\author{}
\date{}
\maketitle

\begin{center}
Naseem Abbas\textsuperscript{1},
Vittorio Colao\textsuperscript{1},
Davide Macrì\textsuperscript{2},
Matthieu Barreau\textsuperscript{3},
William Spataro\textsuperscript{1}
\end{center}

\begin{center}
$^{1}${Department of Mathematics and Computer Science, University of Calabria, Via P. Bucci, 30B, 87036 Arcavacata di Rende, CS, Italy.}
\\
$^{2}${Institute for High Performance Computing and Networking, Italian National Research Council, 87036 Arcavacata di Rende, CS, Italy.}
\\
$^{3}${Digital Futures and KTH Royal Institute of Technology, Stockholm, Sweden.}
\end{center}

\begin{abstract}
\begin{abstract}
We propose a semigroup-aware structured kernel neural
operator for source-driven evolution PDEs.
Unlike generic time-conditioned kernel models, the proposed
formulation represents temporal dependence through a latent
exponential propagator \(e^{tA_\theta}\), where \(A_\theta\)
is a learned latent generator. This provides a structured transition from a non-ODE kernel
representation to an ODE-inspired kernel operator whose
latent temporal evolution satisfies the composition
principle of evolution semigroups. For time-dependent forcing terms, a direct
kernel-based evaluation requires repeated quadrature over
delayed kernel evaluations.
To address this computational bottleneck, we derive a
latent-first reformulation of the source term: the forcing
field is projected onto the kernel-induced latent space,
where its accumulated effect is represented as the solution
of a forced linear latent dynamical system.
For a fixed learned kernel, this reformulation is exact.
The resulting system can then be integrated using a
\(\phi_1\)-based exponential integrator, yielding a
source-evaluation path that is structurally aligned with
the learned latent propagator. The resulting framework combines geometry- and
physics-conditioned low-rank kernel factorization with
alternative source-evaluation paths based on either
explicit Duhamel quadrature or latent exponential
integration. Dissipativity constraints on the latent generator and
identity-preserving corrections near the initial time
are optional architectural components.
Numerical experiments on advection-iffusion-reaction (ADR) 
problems compare the non-ODE and semigroup-aware
formulations, evaluate the quadrature and \(\phi_1\)
source paths, and assess predictive accuracy and
computational cost.
\end{abstract}
\end{abstract}

\section{Introduction}

Many scientific and engineering fields, including fluid dynamics, elasticity, thermal analysis, and electrostatics, rely heavily on partial differential equations (PDEs). Their widespread use has motivated the development of classical numerical techniques such as finite difference methods (FDMs) and finite element methods (FEMs), which solve PDEs by discretizing the physical domain into lattices, elements, or meshes. These approaches are robust and well established, but their application to complex geometries often requires costly meshing procedures \cite{hu2018tetrahedral,hu2020fast}. This limitation is particularly relevant in engineering design workflows, where one often needs to evaluate many related PDE instances under changing geometries, boundary conditions, coefficients, or source terms.

Learning-based solvers have recently emerged as alternatives to traditional numerical methods. One line of work learns solution operators that map input functions to output solutions and then evaluates the solution with a single forward pass \cite{li2020fourier,li2020neural,li2023geometry,li2023fourier,li2022transformer,hao2023gnot,alkin2024universal,wang2024latent,wu2024transolver,xiao2023improved}. Another line of work, including physics-informed neural networks, parameterizes the solution of a given PDE instance and optimizes it using governing equations and observations \cite{krishnapriyan2021characterizing,raissi2019physics}. Although these approaches can reduce inference costs compared with classical solvers, several challenges remain. PINN-type methods typically require retraining when the PDE instance changes \cite{raissi2019physics,li2023neural,nam2024solving}, while many neural-operator architectures are trained on fixed or limited families of domains and may struggle when geometry, boundary data, coefficients, and source functions vary simultaneously.

Neural operators provide a functional viewpoint on PDE learning by approximating maps between function spaces rather than maps between fixed finite-dimensional vectors. Two main constructions have been widely explored. The first one follows the branch--trunk paradigm, where two subnetworks encode input functions and spatial coordinates, as in the operator approximation framework of \cite{chen1995universal} and its extension to Deep Operator Networks \cite{lu2021learning}. The second construction introduces learnable integral operators into neural architectures, leading to Fourier, multipole, and graph neural operators \cite{li2020fourier,li2020multipole,kovachki2023neural}. These models have substantially advanced PDE surrogate modeling, including over irregular geometries and varying discretizations. Related approaches also learn Green's functions or Green-type operators for linear PDEs, enabling fast evaluation for varying source or boundary functions within a given setting \cite{boulle2022data,teng2022learning,negi2024learning,melchers2024neural}. However, many of these methods either rely on domain-specific learned Green's functions, require sufficient sampling of the input functions, or do not explicitly encode the temporal composition structure of evolution equations.

In this work, we focus on source-driven autonomous linear
evolution PDEs whose solutions can be represented through
operator propagation.
The solution contains two components: the propagation of
the initial condition and the accumulated effect of the
time-dependent source term.
The former is governed by an evolution semigroup \(T(t)\),
while the latter appears through the Duhamel integral.
This structure motivates a neural operator that uses a
common learned propagator for both contributions and
incorporates an explicit temporal evolution law.

We start from a structured conditional kernel neural
operator and distinguish between a non-ODE time-conditioned
formulation and a semigroup-aware ODE formulation.
In a generic time-conditioned kernel model, the time
variable is provided to the neural network as an additional
input feature.
While expressive, this construction does not explicitly
enforce the identities \(T(0)=I\) and
\(T(t+s)=T(t)T(s)\), which characterize evolution semigroups.
We therefore introduce a latent exponential propagator
\(e^{tA_\theta}\) inside the learned kernel, where
\(A_\theta\) is a learned latent generator.
This yields a kernel representation whose temporal
component is governed by an autonomous linear latent flow
rather than by an unconstrained time feature.
The semigroup identities hold exactly for this latent
flow; their transfer to the reconstructed operators
requires additional encoding--decoding compatibility.

For the source contribution, the factorized kernel permits
an exact rearrangement of the learned Duhamel response.
The forcing field is first projected into latent space,
where its accumulated effect solves a forced linear
dynamical system.
This system can be evaluated by temporal quadrature or
by a \(\phi_1\)-based exponential integrator.
Spatial projection and temporal propagation can thus be
organized separately, and the resulting latent response
can be shared across output locations.
This rearrangement preserves the continuous learned
operator, whereas the choice of temporal discretization
introduces a separate numerical approximation.

The proposed framework is evaluated on
advection--diffusion--reaction problems, where the solution
depends on diffusion, advection, decay, initial data,
boundary information, geometry, and source terms.
We compare the non-ODE and semigroup-aware formulations
and evaluate explicit Duhamel quadrature against latent
exponential integration.
Across independently trained models, predictive accuracy
depends on both the learned representation and its
numerical evaluation.
To distinguish these effects, we also compare source
integrators using identical trained checkpoints.
For temporally localized sources, adequate temporal
sampling remains necessary independently of the chosen
integration scheme.

The main contributions of this work are summarized
as follows:
\begin{enumerate}
    \item
    We introduce a semigroup-aware extension of structured
    conditional kernel neural operators, replacing generic
    time conditioning with a latent exponential propagator
    \(e^{tA_\theta}\).

    \item
    We provide an operator-factorization interpretation
    that distinguishes exact latent temporal composition
    from composition of the reconstructed operators and
    identifies the associated encoding--decoding defect.

    \item
    We construct predictions by propagating the initial
    condition and accumulating time-dependent source
    contributions through a common learned propagator
    in the Duhamel formula.

    \item
    We derive an exact latent-first reformulation of the
    learned Duhamel contribution and establish discrete
    equivalence with kernel evaluation when identical
    quadrature rules and source values are used.

    \item
    We implement a \(\phi_1\)-based exponential-integrator
    path for the forced latent dynamics and compare its
    left-endpoint and midpoint forcing variants with
    explicit Duhamel quadrature.

    \item
    We evaluate predictive accuracy, semigroup residuals,
    and computational cost, using comparisons between
    independently trained models and source-integrator
    comparisons on identical trained checkpoints.
\end{enumerate}

\subsection{Related Work}

Advection--diffusion--reaction (ADR) equations provide a classical mathematical framework for modeling transport, spreading, and local reaction phenomena. They have been used in several scientific applications, including the analysis of large-scale fish movement \cite{adam2004use}, ecological dispersal processes \cite{cosner2014reaction}, and turbulent flame propagation in astrophysical settings \cite{khokhlov1995propagation}. In these applications, the ADR model is typically specified through problem-dependent physical coefficients and expert-designed assumptions. In contrast, our objective is not to hand-design an ADR model, but to learn a source-driven evolution operator associated with a family of ADR-type problems.

Physics-informed neural networks (PINNs) \cite{raissi2019physics} approximate the solution of a PDE by parameterizing it with a neural network and optimizing the network through losses derived from the governing equations, boundary conditions, and observations. Their implementation is facilitated by automatic differentiation tools available in modern deep learning frameworks \cite{paszke2017automatic,bradbury2018jax}. However, PINNs usually solve one PDE instance at a time. When the source term, boundary condition, physical coefficients, or computational domain change, the model often needs to be retrained or re-optimized. This makes PINN-type approaches less suitable when the goal is to learn a reusable solution operator over a family of related PDE instances.

Neural operators address this limitation by learning mappings between function spaces. Graph Neural Operators (GNOs) \cite{li2020neural} approximate integral kernels through local graph-based message passing, while Fourier Neural Operators (FNOs) \cite{li2020fourier} learn global integral operators in the spectral domain using the fast Fourier transform. Subsequent works combine local and global representations to handle irregular geometries and more complex physical settings \cite{li2023geometry,li2023fourier}. Transformer-based operator-learning models have also been proposed to capture interactions between mesh or point samples through attention mechanisms \cite{li2023scalable,li2022transformer,hao2023gnot,xiao2023improved,vaswani2017attention}. These architectures provide powerful surrogate models for PDE solution operators, but their temporal dependence is often represented through generic conditioning or learned input-output mappings rather than through an explicit semigroup-consistent propagation mechanism.

A related line of work focuses on learning Green's functions or Green-type operators for linear PDEs. Boulle et al. \cite{boulle2022data} learn Green's functions from forcing--solution pairs using rational neural networks \cite{boulle2020rational}. Teng et al. \cite{teng2022learning} extend this idea by approximating Dirac delta functions with Gaussian distributions, enabling the evaluation of solutions under varying source and boundary functions. Negi et al. \cite{negi2024learning} use radial-basis-function neural networks to approximate free-space Green's functions. More recently, Neural Green's Operators have been introduced as finite-dimensional approximations of parametric Green's operators \cite{melchers2024neural}. These approaches are closely related to our kernel-operator viewpoint. However, many of them learn Green's functions tied to specific domains or rely on fixed trial and test representations, whereas our formulation learns a conditional low-rank kernel whose components are adapted through geometry, physics, and problem-context embeddings.

The present work also relates to semigroup-aware operator learning, continuous-depth models, and Koopman-inspired latent dynamics. Semigroup-aware operator learning explicitly accounts for temporal composition and variable propagation intervals \cite{chen2023deep}. Continuous-depth and latent neural dynamics provide a natural way to represent temporal evolution through learned flows \cite{sander2024deeper,xiong2024koopman}. Koopman-based approaches similarly seek latent representations in which nonlinear dynamics can be approximated by linear evolution \cite{xiong2024koopman}. At the same time, recent analyses highlight that learned Koopman or spectral representations should be interpreted with care, since finite-dimensional learned generators do not necessarily recover the true physical generator of the underlying PDE \cite{colbrook2024limits,colbrook2022foundations,colbrook2022computing}. We follow this cautious interpretation: in our framework, $A_\theta$ is a learned latent generator used to construct a structured propagator, not an identified differential operator of the original PDE.

Our work lies at the intersection of these directions. Starting from a structured conditional kernel operator, we replace generic time conditioning with a latent exponential propagator $e^{tA_\theta}$, yielding a semigroup-aware kernel representation. We then use the Duhamel formula to separate the propagation of the initial condition from the contribution of time-dependent source terms. For the latter, we compare an explicit delayed-kernel quadrature with a latent-first reformulation in which the source is projected into the learned latent space and accumulated through a forced linear latent system. This formulation enables a $\phi_1$-based exponential-integrator path while preserving a clear distinction between the learned latent propagation mechanism and the true PDE generator.

\section{Problem statement}
\label{sec:prob}

\subsection{Source-driven advection--diffusion--reaction dynamics}
\label{sec:prob:adr}

We study source-driven advection--diffusion--reaction (ADR)
equations, motivated by contaminant transport.
Advection carries the contaminant with the flow, diffusion
describes spatial spreading, and first-order decay accounts
for attenuation. Time-dependent sources represent subsequent
releases whose timing and spatial distribution determine
the resulting concentration field
\cite{ZhengBennett2002,Wexler1992}.

Specifically, we consider
\begin{equation}
\label{eq:adr}
\begin{cases}
\partial_t u
=
D\Delta u-\mathbf{v}\cdot\nabla u-\lambda u+f,
& \text{in }\Omega\times(0,T),\\
u(\cdot,0)=u_0,
& \text{in }\Omega,
\end{cases}
\end{equation}
with prescribed boundary conditions.
Here $D>0$, $\mathbf{v}\in\mathbb{R}^2$, and $\lambda\geq 0$
are constant diffusion, advection, and decay coefficients,
respectively. The experiments use the normalized domain
$\Omega=(0,1)^2$.

Let $X=L^2(\Omega;\mathbb{R})$, and let $\gamma$ collect
the prescribed physical parameters and boundary information.
For a homogeneous boundary realization generating a
$C_0$-semigroup $\{T_\gamma(t)\}_{t\geq 0}$ on $X$, the
mild solution for $u_0\in X$ and $f\in L^1(0,T;X)$ is
\begin{equation}
\label{eq:duhamel}
u(t)
=
T_\gamma(t)u_0
+
\int_0^t T_\gamma(t-s)f(s)\,\mathrm{d}s,
\end{equation}
where the integral is understood in the Bochner sense
\cite[Section~4.2]{Pazy1983}.
Nonhomogeneous boundary data, when present, require an
additional treatment such as boundary lifting.
The semigroup identities
\[
T_\gamma(0)=I,
\qquad
T_\gamma(t+s)=T_\gamma(t)T_\gamma(s)
\]
describe homogeneous propagation, while
\eqref{eq:duhamel} accumulates source contributions
according to their elapsed propagation times.

\subsection{Operator-learning task}
\label{sec:prob:operator_task}

The objective is to approximate the solution map across
initial conditions, source histories, and prescribed
problem contexts:
\begin{equation}
\label{eq:operator_learning_task}
\mathcal{G}_{\theta}(u_0,f,\gamma)(t,\mathbf{x})
\approx u(\mathbf{x},t).
\end{equation}
For fixed $\gamma$, this map is linear in $(u_0,f)$.
Our construction preserves this superposition structure
by using a common learned propagator for both terms
in~\eqref{eq:duhamel}. It combines a finite-rank spatial
kernel with autonomous linear dynamics in a latent space.

\section{Methods}
\label{sec:methods}

\subsection{Low-rank kernel representation}
\label{sec:methods:low_rank}

The spatial factorization is motivated by the
singular-value decomposition of the positive-time
evolution operator.
Assume that, for $t>0$, $T_\gamma(t)$ admits a kernel
$K_\gamma(t,\cdot,\cdot)\in L^2(\Omega\times\Omega)$.
This property follows, for example, from bounded
positive-time heat kernels on a finite-measure domain;
standard Dirichlet and Neumann ADR realizations on
bounded Lipschitz domains satisfy suitable Gaussian
bounds \cite{Ouhabaz2004}.

The operator $T_\gamma(t)$ is then Hilbert--Schmidt
\cite[Theorem~1.6.1]{Arendt2006}, and its kernel admits
the Schmidt expansion
\begin{equation}
\label{eq:evolution_kernel_svd}
\begin{aligned}
K_\gamma(t,\mathbf{x},\boldsymbol{\xi})
=
\sum_{j\geq 1}
&\sigma_j(t;\gamma)
\varphi_j(t,\mathbf{x};\gamma)\\
&\cdot\chi_j(t,\boldsymbol{\xi};\gamma)
\end{aligned}
\end{equation}
in $L^2(\Omega\times\Omega)$.
Here $\sigma_j$ are the singular values, and
$\varphi_j,\chi_j$ are the left and right singular
functions. The rank-$R$ truncation satisfies
\[
\begin{aligned}
\|T_\gamma(t)-T_{\gamma,R}(t)\|_{\mathcal{L}(X)}
&=\sigma_{R+1}(t;\gamma),\\
\|K_\gamma(t)-K_{\gamma,R}(t)\|_{L^2(\Omega\times\Omega)}^2
&=\sum_{j>R}\sigma_j(t;\gamma)^2.
\end{aligned}
\]
These classical identities justify separable spatial
approximation when the singular-value tail is small
\cite{Simon2005}.
They do not imply a common finite-dimensional temporal
generator: both singular systems generally depend on
time and need not be eigenfunctions of the PDE operator.

We parameterize a separable kernel using neural features
and context-dependent coupling matrices.
Let $E$ be the feature dimension and $R\leq E$ the
coupling rank. A time-conditioned baseline is
\begin{equation}
\label{eq:time_conditioned_kernel}
\begin{aligned}
K_{\theta}^{\mathrm{TC}}
(t,\mathbf{x},\boldsymbol{\xi};\gamma)
={}&
\Phi_{\theta}^{\mathrm{TC}}(t,\mathbf{x})
M_{A,\theta}(\gamma)\\
&\cdot
M_{B,\theta}(\gamma)^\top
\psi_{\theta}(\boldsymbol{\xi})^\top,
\end{aligned}
\end{equation}
where
$\Phi_{\theta}^{\mathrm{TC}},\psi_\theta
\in\mathbb{R}^{1\times E}$ and
$M_{A,\theta},M_{B,\theta}\in\mathbb{R}^{E\times R}$.
The coupling matrices are generated by a hypernetwork
conditioned on $\gamma$
\cite{HaEtAl2017}.
This construction follows the separable parameterization
of low-rank neural operators \cite{kovachki2023neural}.
Its factors are not constrained to be singular functions,
so SVD optimality is a motivation rather than a guarantee
for the learned representation.

\subsection{Latent semigroup propagation}
\label{sec:methods:semigroup_kernel}

We replace direct time conditioning by an autonomous
latent flow:
\begin{equation}
\label{eq:kernel_factorization}
\begin{aligned}
K_{\theta}(t,\mathbf{x},\boldsymbol{\xi};\gamma)
={}&
\Phi_{\theta}(\mathbf{x})e^{tA_{\theta}}
M_{A,\theta}(\gamma)\\
&\cdot
M_{B,\theta}(\gamma)^\top
\psi_{\theta}(\boldsymbol{\xi})^\top,
\end{aligned}
\end{equation}
where $\Phi_{\theta}(\mathbf{x})\in\mathbb{R}^{1\times E}$
is a spatial decoder and
$A_{\theta}\in\mathbb{R}^{E\times E}$ is a learned,
time-independent latent generator.
The spatial features are assumed square-integrable.
The kernel has rank at most $R$ at each fixed time.

The exponential is the solution operator of
$\dot z=A_\theta z$ and satisfies
\begin{equation}
\label{eq:latent_semigroup_identities}
e^{0A_\theta}=I_E,
\qquad
e^{(t+s)A_\theta}=e^{tA_\theta}e^{sA_\theta}.
\end{equation}
Thus, temporal predictions share a common latent
evolution law, rather than being specified independently
through a time-conditioned neural encoder.
The matrix $A_\theta$ represents reduced dynamics and
is not identified with the differential generator
of the ADR equation.

Define the encoding and decoding maps by
\begin{equation}
\label{eq:encoding_decoding_maps}
\begin{aligned}
\mathcal{E}_{\theta,\gamma}g
&=
M_{A,\theta}(\gamma)M_{B,\theta}(\gamma)^\top
\int_\Omega
\psi_\theta(\boldsymbol{\xi})^\top
g(\boldsymbol{\xi})\,\mathrm{d}\boldsymbol{\xi},\\
(\mathcal{P}_\theta z)(\mathbf{x})
&=\Phi_\theta(\mathbf{x})z.
\end{aligned}
\end{equation}
The reconstructed propagator is
\begin{equation}
\label{eq:encoded_latent_propagator}
\widehat T_\theta(t;\gamma)
=
\mathcal{P}_\theta e^{tA_\theta}
\mathcal{E}_{\theta,\gamma}.
\end{equation}

The latent identities
\eqref{eq:latent_semigroup_identities} do not automatically
transfer to the reconstructed operators.
Suppressing $\theta,\gamma$ and writing $G=A_\theta$,
direct multiplication gives
\begin{equation}
\label{eq:decoded_semigroup_defect}
\begin{aligned}
\widehat T(t)\widehat T(s)-\widehat T(t+s)
=
\mathcal{P}e^{tG}
(\mathcal{E}\mathcal{P}-I_E)
e^{sG}\mathcal{E}.
\end{aligned}
\end{equation}
Moreover, $\widehat T(0)=\mathcal{P}\mathcal{E}$
is a finite-rank reconstruction operator.
Accordingly, ``semigroup-aware'' refers to the latent
propagation law; exact composition after reconstruction
requires additional encoding--decoding compatibility
on the relevant latent states.

Latent growth can be controlled through the symmetric
part of $A_\theta$. In particular,
\[
\frac{A_\theta+A_\theta^\top}{2}\preceq-\eta I_E,
\qquad \eta\geq 0,
\]
implies
\[
\|e^{tA_\theta}\|_2\leq e^{-\eta t}.
\]
This is a sufficient stability condition, obtained by
differentiating the squared Euclidean norm of a latent
trajectory and applying Gronwall's inequality.

\subsection{Duhamel representation in latent space}
\label{sec:methods:kernel}

Using the same propagator for the initial state and
the forcing, define
\begin{equation}
\label{eq:initial_condition_path}
u_{\mathrm{IC},\theta}(t)
=
\widehat T_\theta(t;\gamma)u_0
\end{equation}
and
\begin{equation}
\label{eq:duhamel_full}
\begin{aligned}
u_{\mathrm{src},\theta}(t,\mathbf{x})
=
\int_0^t\int_\Omega
&K_\theta(t-s,\mathbf{x},\boldsymbol{\xi};\gamma)\\
&\cdot f(s,\boldsymbol{\xi})
\,\mathrm{d}\boldsymbol{\xi}\,\mathrm{d}s.
\end{aligned}
\end{equation}
The prediction is
\begin{equation}
\label{eq:solution_decomposition}
u_\theta(t)
=
u_{\mathrm{IC},\theta}(t)+u_{\mathrm{src},\theta}(t).
\end{equation}

For a fixed context, introduce the projected source
\begin{equation}
\label{eq:qR_def}
q_R(s)
=
M_{B,\theta}(\gamma)^\top
\int_\Omega
\psi_\theta(\boldsymbol{\xi})^\top
f(s,\boldsymbol{\xi})\,\mathrm{d}\boldsymbol{\xi}
\end{equation}
and
\begin{equation}
\label{eq:bE_def}
b_E(s)
=
M_{A,\theta}(\gamma)q_R(s)
=
\mathcal{E}_{\theta,\gamma}f(s).
\end{equation}
Since $f\in L^1(0,T;X)$ and the encoder is bounded,
$b_E\in L^1(0,T;\mathbb{R}^{E})$.
Factoring out the terms independent of the spatial
integration variable gives the exact identity
\begin{equation}
\label{eq:duhamel_latent}
u_{\mathrm{src},\theta}(t,\mathbf{x})
=
\Phi_\theta(\mathbf{x})
\int_0^t e^{(t-s)A_\theta}b_E(s)\,\mathrm{d}s.
\end{equation}
Equivalently,
\begin{equation}
\label{eq:latent_ode}
\dot z_{\mathrm{src}}
=
A_\theta z_{\mathrm{src}}+b_E,
\qquad z_{\mathrm{src}}(0)=0,
\end{equation}
with
$u_{\mathrm{src},\theta}=\mathcal{P}_\theta z_{\mathrm{src}}$.
This is a forced linear system; the spatial reduction
does not change the continuous learned operator.

The use of a common propagator also connects homogeneous
approximation accuracy to the forced prediction.
Writing
$\mathcal{R}_\theta(\tau)
=T_\gamma(\tau)-\widehat T_\theta(\tau;\gamma)$,
subtraction of the two Duhamel representations yields
\begin{equation}
\label{eq:forced_prediction_error}
\begin{aligned}
\|u(t)-u_\theta(t)\|_X
&\leq
\|\mathcal{R}_\theta(t)u_0\|_X\\
&\quad+
\int_0^t
\|\mathcal{R}_\theta(t-s)f(s)\|_X\,\mathrm{d}s.
\end{aligned}
\end{equation}
This estimate concerns the continuous surrogate;
numerical integration and source reconstruction
introduce additional errors.

\subsection{Quadrature and computational reduction}
\label{sec:methods:latent}

Let $\{(\boldsymbol{\xi}_n,w_n)\}_{n=1}^{N_s}$ be a
spatial quadrature rule and
$\{(s_j,\omega_j)\}_{j=1}^{N_q}$ a Gauss--Legendre rule
mapped to $[0,t]$.
Write $f_h(s_j,\boldsymbol{\xi}_n)$ for the source values
used numerically, obtained either directly or from
temporal reconstruction of available snapshots.
The direct Duhamel approximation is
\begin{equation}
\label{eq:duhamel_quadrature}
\begin{aligned}
\widetilde u_{\mathrm{src},\theta}^{\mathrm q}
(t,\mathbf{x})
=
\sum_{j=1}^{N_q}\omega_j
\sum_{n=1}^{N_s}w_n\,
&K_\theta(t-s_j,\mathbf{x},\boldsymbol{\xi}_n;\gamma)\\
&\cdot f_h(s_j,\boldsymbol{\xi}_n).
\end{aligned}
\end{equation}
The initial-state contribution is evaluated with the
same spatial rule.

The separable kernel permits spatial contraction
before temporal integration. Define
\begin{equation}
\label{eq:discrete_source_projection}
\begin{aligned}
\widetilde q_{R,j}
&=
M_{B,\theta}(\gamma)^\top
\sum_{n=1}^{N_s}
w_n\psi_\theta(\boldsymbol{\xi}_n)^\top
f_h(s_j,\boldsymbol{\xi}_n),\\
\widetilde b_{E,j}
&=M_{A,\theta}(\gamma)\widetilde q_{R,j},
\end{aligned}
\end{equation}
and
\begin{equation}
\label{eq:latent_duhamel_quadrature}
\widetilde z_{\mathrm{src}}^{\mathrm q}(t)
=
\sum_{j=1}^{N_q}
\omega_j e^{(t-s_j)A_\theta}\widetilde b_{E,j}.
\end{equation}

\begin{proposition}
\label{prop:equivalence}
For the factorization~\eqref{eq:kernel_factorization},
identical quadrature rules and source values give
\[
\widetilde u_{\mathrm{src},\theta}^{\mathrm q}
(t,\mathbf{x})
=
\Phi_\theta(\mathbf{x})
\widetilde z_{\mathrm{src}}^{\mathrm q}(t).
\]
The identity holds independently of quadrature accuracy.
\end{proposition}

\begin{proof}
Substitute~\eqref{eq:kernel_factorization} into
\eqref{eq:duhamel_quadrature}, factor out
$\Phi_\theta(\mathbf{x})$, and group the spatial sum
into $\widetilde q_{R,j}$.
All sums are finite, so the rearrangement is exact.
\end{proof}

Thus, kernel and latent evaluation of the same quadrature
have identical discretization errors in exact arithmetic.
The latent vector is independent of the output location
and can be reused to reconstruct the field at all
requested spatial points, without assembling the full
space--time kernel array.

For $B$ problem instances, $N_s=N^2$ spatial nodes,
and $n_t$ available source snapshots per instance,
assembling the projected features
$\psi_\theta(\boldsymbol{\xi}_n)M_{B,\theta}(\gamma)$
costs $\mathcal{O}(BN_sER)$.
After this preparation, source projection and lifting
cost
\begin{equation}
\label{eq:source_projection_complexity}
\mathcal{O}\!\left(Bn_t(N_sR+ER)\right).
\end{equation}
These counts exclude neural feature evaluation and
hypernetwork generation. Reuse requires fixed network
parameters, contexts, and spatial nodes.

If temporal reconstruction is linear, projection can
be performed once per available snapshot and interpolation
subsequently carried out in latent space.
Otherwise, each newly queried source field must be
projected. Temporal quadrature and propagation then
involve $E$-dimensional vectors and matrix functions,
while spatial resolution enters input projection and
output reconstruction. Different time lags generally
require different exponentials, but the same exponential
can be shared across spatial queries.

The forced system~\eqref{eq:latent_ode} also supports
exponential time integration
\cite{HochbruckOstermann2010}, including the
$\phi_1$-based schemes described in
Section~\ref{sec:methods:etd}.
Such schemes provide a different temporal discretization
of the same continuous latent response; their agreement
with Gauss quadrature is not implied by
Proposition~\ref{prop:equivalence}.

\subsection{Exponential integration of the latent source}
\label{sec:methods:etd}

The forced latent system~\eqref{eq:latent_ode} can also
be evaluated by exponential time stepping
\cite{HochbruckOstermann2010}.
Let $0=t_0<\cdots<t_M=T$, with
$h_n=t_{n+1}-t_n$, and let $b_h$ denote the numerically
reconstructed and spatially projected source.
The variation-of-constants formula on one time interval is
\begin{equation}
\label{eq:latent_step_exact}
\begin{aligned}
z_{\mathrm{src}}(t_{n+1})
={}&e^{h_nA_\theta}z_{\mathrm{src}}(t_n)\\
&+\int_0^{h_n}
e^{(h_n-\tau)A_\theta}b_h(t_n+\tau)\,\mathrm{d}\tau.
\end{aligned}
\end{equation}
Here the response is defined relative to the prescribed
numerical source $b_h$.

Define the entire matrix function
\begin{equation}
\label{eq:phi1_definition}
\phi_1(Z)
=
\sum_{k=0}^{\infty}\frac{Z^k}{(k+1)!}
=
\int_0^1 e^{(1-\sigma)Z}\,\mathrm{d}\sigma.
\end{equation}
This definition requires no matrix inverse and gives
$\phi_1(0)=I_E$.
Freezing the forcing at the left endpoint yields
\begin{equation}
\label{eq:etd_left}
z_{n+1}
=
e^{h_nA_\theta}z_n
+
h_n\phi_1(h_nA_\theta)b_h(t_n),
\qquad z_0=0.
\end{equation}
We refer to this update as the standard $\phi_1$ scheme.
The midpoint variant instead uses
\begin{equation}
\label{eq:etd_midpoint}
\begin{aligned}
z_{n+1}
={}&e^{h_nA_\theta}z_n\\
&+h_n\phi_1(h_nA_\theta)
b_h\!\left(t_n+\frac{h_n}{2}\right).
\end{aligned}
\end{equation}
Both updates integrate the homogeneous dynamics exactly
and are exact for forcing that is constant on each step,
provided the matrix functions are evaluated exactly.

For fixed $A_\theta$ and sufficiently smooth $b_h$,
the left-endpoint and midpoint schemes have classical
global temporal orders one and two, respectively.
These statements concern time discretization of the
prescribed latent system; they do not include source
reconstruction or matrix-function approximation errors,
and do not assert error bounds uniform in
$\|A_\theta\|$.

At an observation time reached by the integration grid,
the source prediction is reconstructed as
\[
u_{\mathrm{src},\theta}^{\mathrm{ETD}}
(t_n,\mathbf{x})
=
\Phi_\theta(\mathbf{x})z_n.
\]
The latent trajectory can therefore be shared across
all requested spatial locations.
Gauss quadrature and exponential time stepping
approximate the same continuous latent response,
but generally produce different discrete predictions;
Proposition~\ref{prop:equivalence} concerns only
rearrangements of an identical quadrature rule.

\section{Experimental Setup}

\subsection{Dataset and Training Configuration}

The fair benchmark used the following fixed protocol:

\begin{center}
\scriptsize
\begin{longtable}{@{}>{\raggedright\arraybackslash}p{0.36\textwidth}>{\raggedright\arraybackslash}p{0.58\textwidth}@{}}
\toprule
Item & Configuration \\
\midrule
Data cache & \texttt{cache\_structured\_fair\_seed20260710\_nt10} \\
Cache seed & \texttt{20260710} \\
Training / validation / held-out samples & 6000 / 1200 / 1200 \\
Model seeds & 7, 42, 123, 2024, 2025 \\
Maximum epochs & 1000 \\
Batch size & 16 \\
Training points per sample & 16 \\
Learning rate & \(10^{-3}\) \\
Validation frequency & every 50 epochs \\
Validation sampling seed & \texttt{20260711} \\
Rank & 96 \\
Spatial quadrature resolution & 20 \\
Temporal quadrature nodes & 8 \\
Curriculum epochs & 500 \\
Reconstruction loss & Smooth L1 \\
Source loss weight & 30 \\
Held-out points per sample & 128 \\
Primary ODE embedding dimension & 128 \\
Parameter-matched baseline embedding & 208 \\
ODE generator & non-dissipative \\
Identity skip & disabled \\
Fourier features & disabled \\
\(\phi_1\) order & 8 \\
\(\phi_1\) substep threshold & 0.75 \\
Maximum \(\phi_1\) substeps & 16 \\
\bottomrule
\end{longtable}
\end{center}

The cache was generated once and shared across all model seeds. Validation selected \texttt{best\_model.pt}; the held-out \texttt{test.npz} split was consumed only after model selection. The benchmark recorded the Git commit, worktree state, Python and PyTorch versions, GPU information, installed packages, model metadata, and all required artifacts.

Four independently trained primary arms were evaluated:

\begin{enumerate}
\item non-ODE baseline with Gauss--Legendre quadrature;
\item parameter-matched non-ODE baseline with Gauss--Legendre quadrature;
\item semigroup-aware ODE model with Gauss--Legendre quadrature;
\item semigroup-aware ODE model with the standard \(\phi_1\) source integrator.
\end{enumerate}

Integrator swaps were then evaluated using identical ODE-\(\phi_1\) checkpoints. Consequently, differences between swap rows cannot be attributed to training randomness or to a different learned generator.

\subsection{Evaluation Metrics}

Held-out relative L2 error is reported as a percentage. Separate values are computed for initial-condition and source samples, together with an aggregate value. Lower values are better.

Semigroup consistency is measured using

\[
r(t,s;u_0)=
\frac{\lVert P(t+s)u_0-P(s)[P(t)u_0]\rVert_2}
{\lVert P(t+s)u_0\rVert_2+\varepsilon}\cdot
\]

The reported semigroup value is dimensionless and lower is better. Scenario L2 is the mean relative L2 error across the 23 manually constructed evaluation scenarios. Timing is end-to-end milliseconds per observation and includes both conditioning and propagation. Values are mean \(\pm\) sample standard deviation over five seeds.

\section{Results and Discussion}

\subsection{Comparative Performance of Kernel Models}

\begin{center}
\scriptsize
\resizebox{\textwidth}{!}{%
\begin{tabular}{@{}llllllll@{}}
\toprule
Model & Integrator & Held-out L2 (\%) & Scenario L2 (\%) & Semigroup residual & Parameters (M) & Train time (h) & End-to-end (ms/obs) \\
\midrule
Baseline & Quadrature & 29.135 \(\pm\) 1.843 & 25.173 \(\pm\) 1.774 & 0.361 \(\pm\) 0.034 & 3.291 & 1.615 \(\pm\) 0.967 & 0.011 \(\pm\) 0.006 \\
Baseline, parameter matched & Quadrature & 28.111 \(\pm\) 2.915 & 25.058 \(\pm\) 3.232 & 0.376 \(\pm\) 0.041 & 5.435 & 1.696 \(\pm\) 1.196 & 0.012 \(\pm\) 0.006 \\
ODE & Quadrature & 26.098 \(\pm\) 1.319 & 22.766 \(\pm\) 1.465 & \textbf{0.270 \(\pm\) 0.036} & 5.425 & 7.022 \(\pm\) 3.736 & 0.045 \(\pm\) 0.027 \\
ODE & \(\phi_1\) & \textbf{25.743 \(\pm\) 0.880} & \textbf{22.636 \(\pm\) 0.542} & 0.310 \(\pm\) 0.026 & 5.425 & 6.178 \(\pm\) 1.652 & 0.053 \(\pm\) 0.023 \\
\bottomrule
\end{tabular}%
}
\end{center}

The ODE-quadrature model reduced held-out error by 3.037 percentage points, or 10.4\% relative, compared with the standard baseline. Relative to the parameter-matched baseline, it reduced held-out error by 2.013 points, or 7.2\% relative. Its semigroup residual was 25.2\% lower than that of the standard baseline and 28.2\% lower than that of the parameter-matched baseline.

The ODE-\(\phi_1\) model achieved the lowest primary mean held-out and scenario errors. However, the independently trained ODE arms differ slightly in their initial-condition accuracy, and the same-checkpoint analysis below shows that this aggregate advantage should not be attributed to the standard \(\phi_1\) source integrator.

\subsection{Initial-Condition and Source Error Analysis}

\begin{center}
\scriptsize
\begin{longtable}{@{}>{\raggedright\arraybackslash}p{0.22\textwidth}>{\raggedright\arraybackslash}p{0.23\textwidth}>{\raggedright\arraybackslash}p{0.23\textwidth}>{\raggedright\arraybackslash}p{0.24\textwidth}@{}}
\toprule
Model & IC L2 (\%) & Source L2 (\%) & Total L2 (\%) \\
\midrule
Baseline + quadrature & 29.396 \(\pm\) 1.859 & 21.855 \(\pm\) 1.406 & 29.135 \(\pm\) 1.843 \\
Parameter-matched baseline + quadrature & 28.371 \(\pm\) 2.937 & 20.798 \(\pm\) 2.354 & 28.111 \(\pm\) 2.915 \\
ODE + quadrature & 26.317 \(\pm\) 1.333 & \textbf{20.063 \(\pm\) 0.933} & 26.098 \(\pm\) 1.319 \\
ODE + \(\phi_1\) & \textbf{25.939 \(\pm\) 0.886} & 20.427 \(\pm\) 0.740 & \textbf{25.743 \(\pm\) 0.880} \\
\bottomrule
\end{longtable}
\end{center}

The ODE-quadrature arm reduced source error by 1.792 percentage points (8.2\% relative) compared with the standard baseline and by 0.736 points (3.5\% relative) compared with the parameter-matched baseline. The lower total error of the independently trained ODE-\(\phi_1\) arm is driven by its lower initial-condition error rather than by its source integrator. The total metric is an energy-weighted aggregate, not the arithmetic mean of the IC and source percentages.

\subsection{Source-Integrator Ablation Study}

All rows in the next table use the same learned weights and the same latent generator. Only the inference-time source integrator changes.

\begin{center}
\scriptsize
\resizebox{\textwidth}{!}{%
\begin{tabular}{@{}lllll@{}}
\toprule
Integrator & Source L2 (\%) & Total L2 (\%) & End-to-end (ms/obs) & Interpretation \\
\midrule
Quadrature & 20.045 \(\pm\) 0.808 & 25.732 \(\pm\) 0.882 & approximately 0.040 & Direct reference implementation \\
Standard \(\phi_1\) & 20.427 \(\pm\) 0.740 & 25.743 \(\pm\) 0.880 & 0.053 \(\pm\) 0.023 & Left-endpoint forcing approximation \\
\(\phi_1\)-midpoint & \textbf{20.019 \(\pm\) 0.806} & \textbf{25.731 \(\pm\) 0.882} & \textbf{0.0395 \(\pm\) 0.0002} & Latent exponential propagation with midpoint forcing \\
\bottomrule
\end{tabular}%
}
\end{center}

The standard \(\phi_1\) update increased source error by 0.382 \(\pm\) 0.085 percentage points relative to quadrature and was worse for all five seeds. The exact two-sided paired sign-flip p-value was 0.0625, which is the smallest attainable two-sided value with five non-zero paired observations.

The midpoint update reduced source error by approximately 0.408 points, or 2.0\% relative, compared with the standard \(\phi_1\) update. Relative to quadrature, its average improvement was only 0.0259 points, or approximately 0.13\% relative. It nevertheless produced a lower source error for all five seeds:

\begin{center}
\scriptsize
\begin{longtable}{@{}>{\raggedright\arraybackslash}p{0.22\textwidth}>{\raggedright\arraybackslash}p{0.23\textwidth}>{\raggedright\arraybackslash}p{0.23\textwidth}>{\raggedright\arraybackslash}p{0.24\textwidth}@{}}
\toprule
Seed & Quadrature source L2 (\%) & \(\phi_1\)-midpoint source L2 (\%) & Midpoint minus quadrature (percentage points) \\
\midrule
7 & 20.6218 & 20.6001 & -0.0217 \\
42 & 19.8600 & 19.8299 & -0.0301 \\
123 & 20.8438 & 20.8119 & -0.0319 \\
2024 & 18.7798 & 18.7566 & -0.0232 \\
2025 & 20.1178 & 20.0951 & -0.0227 \\
\bottomrule
\end{longtable}
\end{center}

The mean paired midpoint-minus-quadrature difference was approximately \(-0.0259 \pm 0.0047\) percentage points. The exact two-sided sign-flip p-value is again 0.0625 because only five paired seeds are available. The effect is directionally consistent but practically small. The appropriate conclusion is numerical equivalence at the present scale, not a substantial accuracy advantage over quadrature.

The IC error, semigroup residual, parameter count, and training time are identical across same-checkpoint swaps, as expected: the source integrator does not alter initial-condition propagation or the learned generator.

\subsection{Semigroup Consistency and Generalization}

The ODE-quadrature model achieves the lowest end-to-end semigroup residual, $0.270\pm0.036$, compared with $0.361\pm0.034$ for the standard baseline and $0.376\pm0.041$ for the parameter-matched baseline. These reductions correspond to approximately 25.2\% and 28.2\%, respectively. The ODE-based models also achieve lower mean errors across the 23 evaluation scenarios, supporting improved generalization under the controlled experimental protocol. The semigroup diagnostic assesses the learned physical-space operator and does not establish an exact physical-space semigroup identity.

\subsection{Computational Efficiency and Complexity}
\label{sec:methods:complexity}

The factorized formulation separates spatial source
projection, latent temporal propagation, and output
reconstruction.
This organization is available to both Gauss quadrature
and exponential time stepping, as established by
Proposition~\ref{prop:equivalence}.
The computational benefit of spatial contraction before
temporal integration is therefore distinct from the
choice of temporal integrator.

Consider one source trajectory with fixed network
parameters, context, and spatial quadrature nodes.
Let $N_s$ denote the number of spatial nodes, $n_t$ the
number of available source snapshots, $E$ the latent
dimension, and $R$ the coupling rank.
As described in Section~\ref{sec:methods:latent},
feature preparation and snapshot projection have cost
\[
\mathcal{C}_{\mathrm{proj}}
=
\mathcal{O}\!\left(
N_sER+n_t(N_sR+ER)
\right).
\]
These counts exclude neural feature evaluation,
hypernetwork generation, and temporal source
reconstruction.
When source reconstruction uses scalar temporal weights
common to all spatial nodes, the projected snapshots
can be reconstructed directly in latent space.

Let $L$ denote the number of distinct observation times
and $M$ the total number of requested space--time output
values.
With $Q$ temporal quadrature nodes per observation time,
latent Gauss evaluation has representative cost
\[
\mathcal{C}_{\mathrm{Gauss}}
=
\mathcal{C}_{\mathrm{proj}}
+
\mathcal{O}\!\left(
LQ(E^3+E^2)+ME
\right),
\]
when dense exponentials are constructed separately for
each time lag.
The $E^2$ term accounts for their actions on latent
vectors.
Each accumulated latent response is shared across all
spatial queries at the corresponding observation time;
repeated time lags can permit additional reuse.

For exponential time stepping, let $J$ denote the number
of integration steps and $D\leq J$ the number of distinct
step sizes.
Caching the dense exponential and $\phi_1$ matrices
for each step size gives representative cost
\[
\mathcal{C}_{\mathrm{ETD}}
=
\mathcal{C}_{\mathrm{proj}}
+
\mathcal{O}\!\left(
DE^3+JE^2+ME
\right).
\]
For a uniform grid, $D=1$.
This count assumes that the requested observation times
are included in the computed trajectory.
Additional spatial queries at those times require only
decoding, whereas new observation times may require
further propagation or a separate temporal approximation.

Dense matrix-function construction is not intrinsic
to either latent evaluation path.
Matrix-function actions, structured generators, and
caching can reduce or amortize this cost
\cite{HochbruckOstermann2010}.
Their practical benefit depends on the generator,
latent dimension, required accuracy, and query pattern.

In the same-checkpoint comparison,
$\phi_1$-midpoint achieves a source error of
$20.019\pm0.806\%$, compared with
$20.045\pm0.808\%$ for direct quadrature, with comparable
measured inference time at the current benchmark scale.
These results support the accuracy of the midpoint
source-evaluation path.
A scaling study comparing both integrators with shared
spatial projections and comparable caching is needed
to quantify computational gains at larger query counts.

\section{Conclusions and Future Directions}
The findings demonstrate that the semigroup-aware ODE parameterization improves generalization, temporal consistency, and predictive accuracy compared with the standard baseline. Both ODE-based models achieve lower errors in held-out testing and the 23-scenario evaluation. These improvements cannot be attributed solely to model complexity, as the ODE models contain 5.425 million parameters compared with 5.435 million in the parameter-matched baseline, representing a difference of less than 0.2\%. Furthermore, the ODE-quadrature formulation exhibits superior temporal composition, reducing the semigroup residual by approximately 28.2\% relative to the parameter-matched baseline. Although the standard $\phi_1$ forcing approximation produces higher source errors and longer inference times than direct quadrature in the same-checkpoint comparison, the midpoint forcing modification effectively overcomes these limitations, achieving comparable source and total prediction accuracy with similar computational cost. Overall, the results establish the effectiveness of semigroup-aware latent propagation and demonstrate that the $\phi_1$-midpoint formulation provides a structured representation of source-driven dynamics while maintaining quadrature-level accuracy. However, the present findings support the advantages of the latent reformulation rather than establishing an asymptotic computational speedup.

\subsection{Limitations and Future Work}

The semigroup probe evaluates initial-condition propagation using a fixed conditioning slice, a finite collection of sinusoidal inputs, and a finite set of time pairs. It is an end-to-end consistency diagnostic rather than a direct algebraic verification of the latent matrix exponential. The 23-scenario benchmark contains challenging out-of-distribution cases and is therefore interpreted separately from the disjoint held-out evaluation. The \(\phi_1\)-midpoint comparison was performed after the primary benchmark campaign using the same trained checkpoints. It is consequently reported as a post-hoc same-checkpoint ablation rather than an independently trained primary arm. This distinction leaves the main comparison between the baseline and ODE parameterizations unchanged. A future scaling study should compare latent Gauss quadrature and exponential time stepping using shared
spatial projections and comparable caching strategies.
Varying temporal resolution, spatial query count,
observation times, and latent dimension, while reporting
accuracy, runtime, and peak memory, would quantify
the respective computational benefits.

\bigskip

\textbf{Acknowledgments}

This publication was partially funded by the PhD program in Mathematics and Computer Science at the University of Calabria, Cycle XXXIX, with the support of a scholarship financed by DM
118/2023 (CUP H23C23001760005), based on the PNRR funded by the European Union.\\
V. Colao  acknowledges partial support from the INdAM - GNAMPA Project "Nuovi aspetti teorici e computazionali per le EDO", CUP E53C25002010001, and from ICAR-CNR within the research group “Sustainable and Process-aware Artificial Intelligence” (SPAI), project DIT.AD001.242 “Sistemi Intelligenti e Continuum Quantum Computing”.
\\
\\
\\
\\
\textbf{Data Availability}
\\
{{All the data generated or analyzed during this study are included in this published article.}}
\\
\section*{Declarations}

\textbf{Conflicts of Interest}
{The authors declare that they have no conflict of interest.}

\bibliography{library}
\bibliographystyle{unsrt} 

\end{document}